\documentclass{amsart}

\usepackage{amsmath, amsthm, amssymb}
\usepackage{verbatim}

\usepackage{url}
\usepackage{hyperref}

\renewcommand{\d}{\partial}

\newcommand{\vepsilon}{\varepsilon}
\newcommand{\vphi}{\varphi}

\newcommand{\cE}{\mathcal{E}}

\newcommand{\cF}{\mathcal{F}}

\newcommand{\cM}{\mathcal{M}}

\newcommand{\supp}{\mbox{supp }}

\newcommand{\bC}{\mathbb{C}}

\newcommand{\vpi}{\varpi}

\newtheorem{thm}{Theorem}
\newtheorem{prop}[thm]{Proposition}
\newtheorem{lem}[thm]{Lemma}
\newtheorem{cor}[thm]{Corollary}

\theoremstyle{definition}

\newtheorem{remark}[thm]{Remark}

\numberwithin{thm}{section}
\numberwithin{equation}{section}

\renewcommand{\[}{\begin{equation}}
\renewcommand{\]}{\end{equation}}

\newcommand{\wed}{\wedge}

\title[H\"older continuous subsolution problem, II]{On the H\"older continuous subsolution problem for the complex Monge-Amp\`ere equation, II}

\author[N.-C Nguyen]{Ngoc Cuong Nguyen}

\address{Faculty of Mathematics and Computer Science, Jagiellonian University 30-348 Krak\'ow, \L ojasiewicza 6, Poland}
\email{Nguyen.Ngoc.Cuong@im.uj.edu.pl}
\address{Current address: Department of Mathematics and Center for Geometry and its Applications, Pohang University of Science and Technology, 37673, The Republic of Korea}
\email{cuongnn@postech.ac.kr}


\begin{document}

\maketitle


\begin{abstract} We solve the Dirichlet problem for the complex Monge-Amp\`ere equation on a strictly pseudoconvex with the right hand side being a positive Borel measure which is dominated by the Monge-Amp\`ere measure of a H\"older continuous plurisubharmonic function. If the boundary data is continuous, then the solution is continuous. If the boundary is H\"older continuous, then the solution is also H\"older continuous. In particular, the answer to a question of A. Zeriahi \cite{DGZ16} is always affirmative.\end{abstract}

\section{Introduction}
In this paper we continue our investigation in \cite{Ng17a} of the Dirichlet problem for the complex Monge-Amp\`ere equation in a bounded strictly pseudoconvex domain $\Omega \subset \bC^n$, provided the existence of a H\"older continuous subsolution. We refer the reader to \cite{BT76},  \cite{GKZ08}, and \cite{ko98} for a more detailed historical account on the subject (see also \cite{DDGKPZ14}, \cite{DGZ16}, \cite{DN16}, \cite{DMN17} \cite{ko08}, \cite{hiep} and \cite{viet16} for geometric motivations and applications).

Let $\vphi \in PSH(\Omega) \cap C^{0,\alpha} (\bar\Omega)$ for some $0<\alpha \leq 1$. Assume also that 
\[\notag
 	\vphi = 0 \quad \mbox{on } \d\Omega.
\]
We consider the following set 
\[\notag
	\cM (\vphi, \Omega):= \left\{
\mu \mbox{ is positive Borel measure: } \mu \leq (dd^c\vphi)^n 
	\mbox{ in } \Omega \right\}.\]
We also say that $\vphi$ is  a H\"older continuous subsolution to measures in $\cM(\vphi,\Omega)$. Given $\psi$ a H\"older continuous function on the boundary $\d\Omega$ and a measure $\mu$ in $\cM(\vphi, \Omega)$ we look for a real-valued function $u$ satisfying
\[\label{eq:dirichlet-prob}
\begin{aligned}
&	u\in PSH\cap L^\infty(\Omega),  \\
&	(dd^c u)^n = \mu \quad \mbox{in } \Omega, \\
&	\lim_{z\to x} u(z) = \psi(x) \quad \mbox{for } x\in\d\Omega,
\end{aligned}\]
and 
\[\label{eq:holder}
	u \in C^{0,\alpha'}(\bar\Omega) \quad \mbox{for some } 0< \alpha' \leq 1. \\
\]
The  Dirichlet problem~\eqref{eq:dirichlet-prob} was solved by Ko\l odziej \cite{ko95}  provided that there exists a  bounded plurisubharmonic subsolution. In our setting, the H\"older continuity of $\psi$ on $\d\Omega$ and of $\vphi$ on $\bar\Omega$ are necessary in order to solve the Dirichlet problem \eqref{eq:dirichlet-prob} $\&$ \eqref{eq:holder}. In \cite{Ng17a} this problem is solved under the extra assumptions:
\[\notag
	\psi \equiv 0 \quad \mbox{ and } \quad \int_\Omega (dd^c \vphi)^n <+\infty.
\]
This gave also an affirmative answer to a question of A. Zeriahi \cite[Question 17]{DGZ16} when the subsolution $\vphi$ has finite Monge-Amp\`ere total mass. Our goal now is to remove these extra assumptions. The first main result of this paper is as follows.

\medskip

\noindent{\bf Theorem~A. }{\em 
Let $\psi \in C^0(\d\Omega)$ and $\mu \in \cM(\vphi,\Omega)$. Then, there exists a unique solution $u\in C^0(\bar\Omega)$ to the   Dirichlet problem \eqref{eq:dirichlet-prob}.}
\medskip

This theorem is closely related to a question of S.\,Ko\l odziej \cite[Question 14]{DGZ16} where he asked if one could prove Theorem~A   when the subsolution $\vphi$ is only \em continuous? \rm The question is still open in general. 

The next result gives a necessary and sufficient condition under which a positive Borel measure admitting a H\"older continuous plurisubharmonic potential. In particular, the answer to the above question of A.\,Zeriahi is affirmative.

\medskip

\noindent{\bf Theorem B. } {\em Assume that $\psi$ is H\"older continuous and $\mu \in \cM(\vphi,\Omega)$. Then,  the Dirichlet problem \eqref{eq:dirichlet-prob} $\&$ \eqref{eq:holder} is solvable.}

\medskip

Thanks to this we obtain easily the convexity of the set of Monge-Amp\`ere measures of H\"older continuous plurisubharmonic functions in $\Omega$. Another important consequence is the $L^p$ property. Thus, results in \cite{BKPZ16}, \cite{Cha15a, Cha15b} are special cases of ours. 

\medskip

\noindent{\bf Corollary C.} {\em Let $\mu \in \cM(\vphi,\Omega)$ and $f\in L^p(\Omega, d\mu)$, $p>1$, a nonnegative function. Suppose that $\vphi$ is H\"older continuous plurisubharmonic function on a neighborhood of $\bar\Omega$. Then, $f\mu \in \cM(\tilde\vphi,\Omega)$ for a H\"older continuous plurisubharmonic function $\tilde\vphi$ in $\Omega.$
}

\bigskip
\noindent{\em Acknowledgement. } I am very grateful to S\l awomir Ko\l odziej for many useful discussions. I would like to thank Kang-Tae Kim for his generous support and encouragement. 
The author is supported by  the NRF Grant 2011-0030044 (SRC-GAIA) of The Republic of Korea.

\section{Preliminaries}

In this section we will recall results that are needed in the proofs of Theorems A and B, Corollary~C. If there is no other indication, then the notations in this section will be used for the rest of the paper.

Let $\Omega$ be a bounded strictly pseudoconvex domain in $\bC^n$. Let $\rho\in C^2(\bar\Omega)$ be a strictly plurisubharmonic defining function for $\Omega$. Namely, 
\[\label{eq:defining-fct}
	\Omega = \{\rho < 0\} \quad \mbox{and} \quad d\rho \neq 0 \mbox{ on } \d\Omega.
\]
Let us denote by $\beta = dd^c |z|^2$ the standard K\"ahler form in $\bC^n$. Without loss of generality we may assume that
\[
	dd^c\rho \geq \beta \quad \mbox{on } \bar\Omega.
\]
Throughout the paper the H\"older continuous subsolution $\vphi$ and the associated set of measures $\cM(\vphi, \Omega)$ are defined as in the introduction.  

The following estimate will be very useful for us. For simplicity we write
\[
	\|\cdot\|_\infty := \sup_{\Omega} |\cdot |.
\]
\begin{lem}[B\l ocki \cite{Bl93}] \label{lem:blocki}
Let $v_1,...,v_n, v, h \in PSH \cap L^\infty(\Omega)$ be such that $v_i \leq 0$ for $i =1, ...,n$, and $v\leq h$. Assume that $\lim_{z\to \d\Omega} [h(z) - v(z)] =0$. Then, for an integer $1\leq k \leq n$,
\[\begin{aligned}
&	\int_{\Omega} (h-v)^k dd^c v_1 \wed \cdots \wed dd^c v_n \\
&\leq k! \|v_1\|_\infty \cdots \|v_k\|_\infty \int_{\Omega} (dd^cv)^k \wed dd^cv_{k+1} \wed \cdots \wed dd^c v_n.
\end{aligned}\]
\end{lem}

Consider also the following  Cegrell class: 
\[ \cE_0 = \left\{ 
v \in PSH \cap L^\infty(\Omega) \biggm|
\begin{aligned}& \lim_{x\to z} v(x) =0 \quad  \forall z\in \d\Omega,  \\
&\mbox{ and } \int_\Omega (dd^c v)^n < +\infty \end{aligned}\right\}.
\]
The Cegrell inequaliy in this class reads:
\begin{lem}[Cegrell \cite{Ce04}] \label{lem:cegrell}
Let $v_1,..., v_n \in \cE_0$. Then, 
\[
	\int_\Omega dd^c v_1 \wed \cdots \wed dd^c v_n \leq 
	\left(\int_\Omega (dd^cv_1)^n\right)^\frac{1}{n} \cdots 
	\left(\int_\Omega (dd^cv_n)^n\right)^\frac{1}{n}.
\]
\end{lem}
We need also to work with a subclass of the Cegrell class:
\[
	\cE_0' := \left\{v\in \cE_0: \int_\Omega (dd^cv)^n \leq 1\right\}.
\]

The decay of the volume of sublevel sets of functions in the class $\cE_0'$ is equivalent to the volume-capacity inequality.  This inequality plays a crucial role in the capacity method due to Ko\l odziej to obtain the \em a priori and stability estimates \rm for weak solutions of complex Monge-Amp\`ere equation. Here the capacity is the Bedford-Taylor capacity and it is defined as follows. For a Borel set $E\subset \Omega$
\[
	cap(E,\Omega):= \sup\left\{\int_E (dd^cw)^n: w\in PSH(\Omega), \; 0\leq w\leq 1\right\}.
\]
In what follows we shall write $cap(E)$ instead of $cap(E,\Omega)$ for simplicity as the domain $\Omega$ is already fixed.

\section{Proof of Theorem A}

In this section we shall prove the following result. 

\begin{prop}\label{prop:vol-cap} Assume that $\mu \in \cM(\vphi, \Omega)$. Then, there exist uniform constants $\alpha_0,C>0$ depending only on $\vphi, \Omega$ such that for every compact set  $K\subset\Omega$, 
\[
	\mu(K) \leq C cap(K) \exp \left(\frac{-\alpha_0}{[cap(K)]^\frac{1}{n}}\right).
\]
\end{prop}

Notice that under the assumption 
$
	\int_\Omega (dd^c\vphi)^n < +\infty
$
a similar inequality, without the factor $cap(K)$ on the right hand side, was proven in \cite{Ng17a}.

\begin{remark} Theorem~A will follows immediately from the proposition and a result of Ko\l odziej \cite[Theorem~5.9]{ko05} as $\mu$ belongs to the class $\cF(A, h)$ with $h = e^{\alpha_0 x}$ and a uniform $A>0$.
\end{remark}

We  will need the following two lemmas. The first one tells us how fast the Monge-Amp\`ere mass of $(dd^c\vphi)^n$ on large sublevel sets goes to infinity.

\begin{lem}\label{lem:cegrell-sublevel}
Let $v\in \cE_0'$. Then, there exists a uniform constant $C$ such that for $s>0$,
\[
	\int_{\{v<-s\}} (dd^c \vphi)^n \leq  \frac{C\|\vphi\|_\infty^n}{s^n}.
\]
\end{lem}

\begin{proof} Set  $v_s := \max\{v, -s\}$. Then, $v_s = v$ on a neighborhood of $\d\Omega$. Moreover,
\[
	v_{s/2} - v \geq \frac{s}{2}  \quad \mbox{on }  \{v< -s\} \subset\subset \Omega.
\]
Therefore,
\[\label{eq:lcs}\begin{aligned}
	\int_{\{v<-s\}} (dd^c\vphi)^n 
&\leq  	\left(\frac{2}{s}\right)^n \int_{\Omega} (v_{\frac{s}{2}} -v)^n (dd^c \vphi)^n \\
&\leq		\frac{2^n n!}{s^n} \|\vphi\|_\infty^n \int_\Omega (dd^cv)^n,
\end{aligned}\]
where the second inequality follows from Lemma~\ref{lem:blocki}.
\end{proof}

On the other hand the volume, with respect to the measure $(dd^c\vphi)^n$, of small sublevel sets of functions in  $\cE_0'$ decays exponentially fast to zero. The H\"older continuity of $\vphi$ is crucially important to prove such an estimate.

\begin{lem} \label{lem:exp-decay}
There exist uniform constants $\tau>0$ and $C>0$ such that for $v\in \cE_0'$ and $s\geq 2$,
\[
	\int_{\{v<-s\}} (dd^c\vphi)^n \leq C e^{- \tau s}.
\]
\end{lem}

\begin{proof} 
With the same notations as in the proof of Lemma~\ref{lem:cegrell-sublevel} we have for $s\geq 2$
\[ \label{eq:exp-decay-A}
\begin{aligned}
	\int_{\{v<-s\}} (dd^c\vphi)^n 
&\leq  	\frac{2}{s} \int_{\Omega} (v_{\frac{s}{2}} -v) (dd^c \vphi)^n\\
&\leq		\int_{\Omega} (v_{\frac{s}{2}} -v) (dd^c \vphi)^n.
\end{aligned}
\]
Let us denote
\[\notag
	S_k := (dd^c\vphi)^k \wed \beta^{n-k},
\]
where $\beta= dd^c |z|^2$ and $0 \leq k\leq n$ is integer. Our first goal is to show that there exist $\alpha_k>0$ and $C>0$ (independent of $v$ and $s$) such that for $v\in \cE_0'$ and $s\geq 1$,
\[\label{eq:holder-cegrell-class}
	\int_\Omega (v_s -v) S_k \leq C \left(\int_\Omega (v_s -v) dV_{2n}\right)^{\alpha_k},
\]
where $v_s = \max\{v, -s\}$.  Indeed, without loss of generality we may assume that
\[\label{eq:small-ass}
	0< \|v_s -v\|_1 < 1/100.
\]
Otherwise, if $\|v_s -v\|_1 =0$, then the inequality trivially holds. If $\|v_s -v\|_1 \geq 1/100$, then we have, using $s\geq 1$, $v\leq 0$ and Lemma~\ref{lem:blocki}, that
\[\begin{aligned}
	\int_\Omega (v_s -v) S_k 
&= 		\int_{\{v<-s\}} (-s-v) S_k \\
&\leq \int_{\Omega} (-v)^k S_k\\
&\leq		C \|\vphi\|_\infty^k.
\end{aligned}\]
This implies the inequality.

Next, under the assumption \eqref{eq:small-ass} we prove the inequality by induction in $k$. The case $k=0$ is obvious. Assume that 
for every integer $m\leq k$ we have 
\[
	\int_\Omega (v_s -v) S_m \leq C \left(\int_\Omega (v_s -v) dV_{2n}\right)^{\alpha_m}.
\]
Then, we need to show that there exists $0< \alpha_{k+1} \leq 1$ such that
\[
	\int_\Omega (v_s -v) S_{k+1} \leq C \left(\int_\Omega (v_s -v) dV_{2n}\right)^{\alpha_{k+1}}.
\]
For simplicity we write 
\[
	S:= (dd^c \vphi)^{k} \wed \beta^{n-k-1}.
\]
Let us still write $\vphi$ to be the H\"older continuous extension of $\vphi$ onto a neighbourhood $U$ of $\bar \Omega$.  Consider the convolution of $\vphi$ with the standard smooth kernel $\chi$, i.e.,
$\chi \in C^\infty_c(\bC^n)$ such that $\chi(z) \geq 0$, $\chi(z)= \chi(|z|)$, $\supp \chi \subset\subset B(0,1) $ and $\int_{\bC^n}\chi (z) dV_{2n} =1$.
Namely, for $z\in U$ and $\delta>0$ small,
\[\begin{aligned}
	\vphi * \chi_t (z) 
&=		 \int_{B(0,1)} \vphi(z-t z')  \chi (z') dV_{2n}(z')  \\
&=		\frac{1}{t^{2n}} \int_{B(z, t)} \vphi (z') \chi \left(\frac{z -z'}{t} \right) dV_{2n} (z').  
\end{aligned}\]
Observe that 
\[\label{eq:holder-convol}
\begin{aligned}
	\vphi * \chi_t (z) - \vphi(z) 
&=	 \int_{B(0,1)} [\vphi(z-t z') -\vphi(z)] \chi (z') dV_{2n}(z')  \\
&	\leq Ct^\alpha,
\end{aligned}\]
and
\[\label{eq:sec-der-convol}
	\left|\frac{\d^2 \vphi * \chi_t}{\d z_j \d \bar z_k} (z)\right| \leq \frac{C \|\vphi\|_\infty}{t^2}.
\]
We first have 
\[ \label{eq:i1+i2}
\begin{aligned}
	\int_\Omega (v_s-v) dd^c \vphi \wed S  
&\leq 	\left| \int_\Omega (v_s -v) dd^c \vphi*\chi_t \wed S \right| \\
&\quad + \left| \int_\Omega (v_s -v) dd^c (\vphi*\chi_t -\vphi) \wed  S\right| \\
&=: I_1+ I_2.
\end{aligned}\]
It follows from \eqref{eq:sec-der-convol} that
\[
	I_1 \leq \frac{C \|\vphi\|_\infty}{t^2} \int_\Omega (v_s -v) S \wed \beta = 
	\frac{C \|\vphi\|_\infty}{t^2} \int_\Omega (v_s -v) S_{k} .
\]
Hence,
\[
	I_1 \leq \frac{C \|\vphi\|_\infty}{t^2} \|v_s -v\|_1^{\alpha_k}.
\]
We turn to the estimate of the second integral $I_2$. By integration by parts
\[\label{eq:i2-int}\begin{aligned}
	\int_\Omega (v_s-v) dd^c (\vphi*\chi_t -\vphi) \wed S  
&= 		\int_{\Omega} (\vphi*\chi_t -\vphi) dd^c (v_s -v) \wed S \\
&= 		\int_{\{v < -\frac{s}{2}\}} (\vphi*\chi_t -\vphi) dd^c (v_s -v) \wed S
\end{aligned}\]
as $v_s = v$ on $ \{v \geq -s\}$. Hence,
\[\begin{aligned}
	I_2 
&\leq 	\int_{\{v<-\frac{s}{2}\}} |\vphi*\chi_t -\vphi| (dd^c v + dd^c v_s) \wed S\\
&\leq 	C t^\alpha  \int_{\{v<-\frac{s}{2}\}} (dd^c v + dd^c v_s) \wed S.
\end{aligned}\]
For the first term of the integral on the right hand side we have 
\[\label{eq:ind-i2-a1}\begin{aligned}
	\int_{\{v<-\frac{s}{2}\}} dd^c v \wed S 
&\leq		 \left(\frac{4}{s}\right)^k \int_{\{v<-\frac{s}{4}\}} (v_{\frac{s}{4}} -v)^k dd^c v \wed S \\
&\leq		\frac{C}{s^k} \int_\Omega (v_{\frac{s}{4}} -v)^k dd^c v \wed (dd^c \vphi)^k \wed \beta^{n-k-1}.
\end{aligned}\]
Applying Lemma~\ref{lem:blocki} we conclude that
\[\label{eq:ind-i2-a2}
	\int_{\Omega} (v_{s/4} -v)^k dd^cv \wed (dd^c\vphi)^k \wed \beta^{n-k-1} \leq C \|\vphi\|_\infty^k \int_\Omega (dd^cv)^{k+1} \wed \beta^{n-k-1}.
\]
Using $dd^c\rho \geq \beta$ (see Preliminaries) and Cegrell's inequality we get that
\[\label{eq:ind-i2-a3}\begin{aligned}
\int_\Omega (dd^cv)^{k+1} \wed \beta^{n-k-1}
&\leq 	\int_\Omega (dd^cv)^{k+1} \wed (dd^c\rho)^{n-k-1}  \\
&\leq		\left(\int_\Omega (dd^cv)^n\right)^\frac{k+1}{n} 
	\left(\int_\Omega (dd^c\rho)^n\right)^\frac{n-k-1}{n}.
\end{aligned}\]
Combining \eqref{eq:ind-i2-a1}, \eqref{eq:ind-i2-a2} and \eqref{eq:ind-i2-a3}  we have for $s\geq 1$,
\[
	\int_{\{v<-s/2\}} dd^c v \wed S \leq C \|\vphi\|_\infty^k.
\]
Notice that $v_s \in \cE_0'$.  The same arguments as above imply that for $s\geq 1$,
\[
	\int_{\{v<-s/2\}} dd^c v_s \wed S \leq C \|\vphi\|_\infty^k.
\]
Thus, altogether we have
\[
	I_1+ I_2 \leq \frac{C \|\vphi\|_\infty}{t^2} \|v_s-v\|_1^{\alpha_k} +  C\|\vphi\|_\infty^k t^{\alpha}.
\]
If we choose 
\[
	t = \|v_s -v\|_1^\frac{\alpha_k}{3}, \quad 
	\alpha_{k+1} = \frac{\alpha \alpha_k}{3},
\]
then the proof of \eqref{eq:holder-cegrell-class} is completed.

We now conclude the proof of the lemma. It follows from \cite[Eq. (2.26)]{Ng17a} and \cite[Lemma~4.1]{ko05} that 
$$
	\int_\Omega (v_s -v) dV_{2n} \leq C e^{-\tau_0 s},
$$
where $\tau_0>0$ and $C>0$ are uniform constants independent of $v$ and $s$. Combining this with \eqref{eq:exp-decay-A} and the inequality \eqref{eq:holder-cegrell-class} for $k=n$  the lemma follows.
\end{proof}

We are ready to prove the main result of this section. 

\begin{proof}[Proof of Proposition~\ref{prop:vol-cap}] 
Let us denote $\nu := (dd^c\vphi)^n$. First, we show that for $v\in \cE_0'$ 
there exist uniform constants $\alpha_1, C>0$ such that 
\[\label{eq:sharp-vol-cap}
	\nu (v<-s) \leq \frac{Ce^{-\alpha_1s}}{s^n} \quad \forall s>0.
\]
Indeed, there are two possibilities either $s\geq 2$ or $s<2$. If $s \geq 2$, then the inequality follows from Lemma~\ref{lem:exp-decay} as  
$$s^n e^{-\tau s/2} \leq \left(\frac{2n}{\tau}\right)^n e^{-n}.$$ 
(We can take $\alpha_1 = \tau/2$). Otherwise, if $0<  s <2$, then  we have $e^{-\alpha_1 s} \geq C.$ Then, the desired inequality follows from Lemma~\ref{lem:cegrell-sublevel}.

To complete the proof of the proposition we use an argument which is inspired by the proofs in \cite{ACKPZ09}. Let $K\subset \Omega$ be compact. Since $\nu$ is dominated by a Monge-Amp\`ere measure of a bounded plurisubharmonic function, it vanishes on pluripolar sets. Hence, we may assume that $K$ is non-pluripolar. Let $h_K^*$ be the relative extremal function of  $K$ with respect to $\Omega$.  
Since $K \subset \Omega$ is compact, it is well-known that 
\[\notag
	\lim_{\zeta \to \d \Omega} h_{K}^*(\zeta)=0.
\]
By \cite[Proposition~5.3]{BT82} we have
\[\notag
	\tau^n:= cap(K, \Omega) = \int_{\Omega} (dd^ch_K^*)^n >0.
\]
Let $0<x<1.$ Since the function $w:= \frac{h_K^*}{\tau}$ satisfies assumptions of the inequality \eqref{eq:sharp-vol-cap}, we have
\[\notag
	\nu (h_K^* < -1 +x) = \nu \left(w < \frac{-1+x}{\tau} \right) \leq C \frac{\tau^n}{\alpha_1^n (1-x)^n}\exp\left({-\frac{\alpha_1(1-x)}{\tau}}\right).
\]
Let $x\to 0^+$, we obtain
\[\label{eq:cap-ineq1}
	\nu(h_K^* \leq -1) \leq  \frac{C}{\alpha_1^n} cap(K,\Omega)\exp\left({\frac{-\alpha_1}{\left[cap(K,\Omega)\right]^\frac{1}{n}}}\right).
\]
Since $h_K = h_K^*$ outside a pluripolar set, we have
\[\label{eq:cap-ineq2}
	\nu(K) \leq \nu(h_K =-1) = \nu(h_K^* =-1) \leq \nu(h_K^* \leq -1).
\]
We combine \eqref{eq:cap-ineq1} and \eqref{eq:cap-ineq2} to finish the proof.
\end{proof}

\section{Proof of Theorem B}

In this section we will prove the H\"older continuity of the solution obtained in Theorem A provided furthermore that the boundary data $\psi$ is H\"older continuous.  Notice that the zero boundary values of the subsolution $\vphi$ is not essential.  We can modify it by adding  an appropriate envelope, similar to \eqref{eq:envelope}, because no condition has been imposed on the total mass of the subsolution.

By Theorem A there exists a unique continuous solution to the Dirichlet problem \eqref{eq:dirichlet-prob}, namely, $u \in PSH(\Omega)\cap C^0(\bar\Omega)$ solving
\[\label{eq:sol-B}
	(dd^c u)^n = \mu, \quad u(z) = \psi (z) \quad \forall z\in \d\Omega.
\]
We are going to  show that $u \in C^{0,\alpha'}(\bar\Omega)$  for some exponent $0<\alpha'\leq 1.$ 

\medskip

{\em Outline of the proof.\rm} Let us sketch  the proof of Theorem~B.
Overall we follow the steps in the proof of \cite{Ng17a} which in turns followed \cite{GKZ08}. Though, we need to consider the problem on an increasing  exhaustive sequence of relatively compact domains in $\Omega$.  Denote for $\delta>0$ small
\[
	\Omega_\delta:= \{z\in \Omega: dist(z,\d\Omega) >\delta\};
\]
and for $z\in \Omega_\delta$ we define
\[\begin{aligned}
&	u_\delta(z) := \sup_{|\zeta| \leq \delta} u(z+\zeta), \\
&	\hat u_\delta(z):= \frac{1}{\sigma_{2n}\delta^{2n}} \int_{|\zeta| \leq \delta} u(z+\zeta) dV_{2n}(\zeta),
\end{aligned}\]
where $\sigma_{2n}$ the volume of the unit ball.

Then, we wish to show that
$$\sup_{\Omega_\delta} (\hat u_\delta - u) \lesssim  \delta^{\vpi}  
$$
for some $0<\vpi \leq 1$. 
Thanks to the H\"older continuity of the boundary data we can extend $\hat u_\delta$ to $\tilde u$ 
 by a gluing process such that the new function is plurisubharmonic on $\Omega$ and equal to $u$ outside $\Omega_\vepsilon$ for some (small) $\vepsilon > \delta$. Moreover, we shall  still have
$$	
	\sup_{\Omega_\delta} (\hat u_\delta - u) 
	\leq \sup_{\Omega} (\tilde u -u) + C \vepsilon^\alpha,
$$
where $\alpha$ is the H\"older exponent of the boundary data $\psi$.
Next, we shall show that
\[\notag
	\int_{\Omega_\vepsilon} (dd^c\vphi)^n \lesssim \frac{1}{\vepsilon^n}.
\]
This estimate enables us to invoke the results of \cite{Ng17a}. It gives a precise quantitative estimate  $\sup_{\Omega} (\tilde u -u)$ in terms of $\delta$ and $\vepsilon$. 
Finally, we can choose $\vepsilon = \delta^{\vpi'}$ with $\vpi'>0$ so small that our desired inequality holds.

\medskip

We now proceed to give details of the argument. 
For the rest of the arguments we fix a small $\delta_0>0$ and consider two parameters $\delta, \vepsilon$ such that
\[\label{eq:eps-del}
0 < \delta \leq \vepsilon < \delta_0.
\]
We may assume that $\psi\in C^{0,2\alpha}(\d\Omega)$, where $0<\alpha\leq 1/2$ (decreasing $\alpha$ if necessary) is the H\"older exponent of the subsolution $\vphi$. Then, we define
\[\label{eq:envelope}
	h(z) = \sup\{v(z) \in PSH(\Omega)\cap C^0(\bar\Omega): h_{|_{\d\Omega}} \leq \psi\}.
\]
It is well-known \cite[Theorem~6.2]{BT76} that $h \in PSH(\Omega)\cap C^{0,\alpha}(\bar\Omega)$ and $h = \psi$ on $\d\Omega$, which is also the solution of the homogeneous Monge-Amp\`ere equation in $\Omega$. Hence, we may assume that 
\[
	\psi \in PSH(\Omega) \cap C^{0,\alpha}(\bar\Omega)\quad \mbox{and}\quad (dd^c\psi)^n \equiv 0.
\]
Thanks to the comparison principle \cite{BT82} we get that
\[\label{eq:boundary-hol}
	\psi + \vphi \leq u \leq \psi \quad \mbox{ on } \bar\Omega.
\]
\begin{lem} \label{lem:boundary-holder}
We have for $z\in \bar\Omega_{\delta}\setminus\Omega_{\vepsilon}$,
\[	u_\delta(z) \leq u(z) + C \vepsilon^\alpha.
\]
In particular, 
\[
	\sup_{\Omega_\delta} (\hat u_\delta - u) \leq \sup_{\Omega_\vepsilon} (\hat u_\delta - u) + C \vepsilon^\alpha.
\]
\end{lem}

\begin{remark} It is important to keep in mind that the uniform constants $C>0$ appeared in the lemma, and many times below are independent of  $\delta$ and  $\vepsilon$. 
\end{remark}

\begin{proof} Fix a point $z\in \bar\Omega_\delta\setminus\Omega_\vepsilon$. Since $u$ is continuous, there is $\zeta_1 \ \in \bC^n$ with $|\zeta_1| \leq \delta$ such that
\[\label{eq:point1}
	u_\delta(z) = u(z + \zeta_1).
\]
Moreover, there exists $\zeta_2\in \bC^n$ with $|\zeta_2| \leq \vepsilon$ such that $z+\zeta_2 \in \d\Omega$. Using this and \eqref{eq:boundary-hol} we get that
\[\begin{aligned} 
	u_\delta(z) - u(z) 
&\leq 	\psi(z+ \zeta_1) - [\psi(z) +\vphi(z)] \\
&= 	[\psi (z+ \zeta_1) - \psi(z)] + [\vphi(z) - \vphi(z+\zeta_2)] \\
&\leq	 	C_1 |\zeta_1|^{\alpha} + C_2|\zeta_2|^{\alpha},	
\end{aligned}\]
where $C_1 = \|\psi\|_{C^{0,\alpha}}, C_2= \|\vphi\|_{C^{0,\alpha}}$. Since $\delta \leq \vepsilon$ we conclude the proof of the first part.

To prove the second part, we observe that  $u \leq \hat u_\delta \leq u_\delta$. Therefore,
\[
	\sup_{\Omega_\delta} (\hat u_\delta - u) \leq \sup_{\Omega_\vepsilon} (\hat u_\delta - u) + \sup_{\Omega_\delta \setminus \Omega_\vepsilon} (u_\delta -u).
\]
Combining this with the first part we get the the second part.
\end{proof}

The lemma above tells us that to obtain the H\"older continuity of the solution $u$ it is enough to get the estimate on the domain $\Omega_\vepsilon$ for $\vepsilon$ being of a small constant compared to  $\delta$.  To achive our goal we will work on the domain $\Omega_\vepsilon$ and keep track of the (negative) exponent of $\vepsilon$.

Recall that
\[\label{eq:omega-epsilon}
\Omega_\vepsilon= \{z\in \Omega: dist(z,\d\Omega) >\vepsilon\}.
\]
We define 
\[\label{eq:d-domain}	 D_{\vepsilon}:= \{\rho(z) < - \vepsilon\},
\]
where $\rho$ is the defining function of $\Omega$ as in \eqref{eq:defining-fct}. 
The following lemma is very similar to Lemma~\ref{lem:cegrell-sublevel}. The main observation is that the domains $D_\vepsilon$ and $\Omega_\vepsilon$ are comparable.

\begin{lem}\label{lem:mass-est} 
Let  $1\leq k \leq n$ be an integer. Let $v\in PSH\cap L^\infty(\Omega)$. Then, 
\[
	\int_{\Omega_\vepsilon} (dd^cv)^k \wed \beta^{n-k} \leq \frac{C  \|v\|_\infty^k}{\vepsilon^k},
\]
where $C$ is independent of $\vepsilon.$
\end{lem}

\begin{proof} Observe that, from Hopf's lemma,
\[ \label{eq:dist-bound} 
|\rho(z)| \geq c_0 dist(z, \d\Omega)
\] 
for a  uniform constant $0 < c_0 \leq 1$. Therefore, 
\[\label{eq:inclusion}
	\Omega_\vepsilon \subset \{\rho(z) < - c_0 \vepsilon\}.
\]
Since $\max\{\rho, -\vepsilon'/2\} - \rho \geq \vepsilon'/2$ with $\vepsilon'= c_0\vepsilon$ on the latter set, it follows  that
\[\begin{aligned}
&	\int_{\Omega_\vepsilon} (dd^cv)^k \wed \beta^{n-k} \\
&\leq		\left(\frac{2}{\vepsilon'}\right)^k\int_{\Omega} \left(\max\{\rho, -\vepsilon'/2\} - \rho\right)^k (dd^cv)^k \wed \beta^{n-k} \\
&\leq		\frac{C\|v\|_\infty^k}{\vepsilon^k}  \int_{\Omega} (dd^c\rho)^k\wed \beta^{n-k},
\end{aligned}\]
where we used Lemma~\ref{lem:blocki} for the second inequality. The last integral is bounded by the $C^2-$smoothness of $\rho$ on $\bar\Omega$.
\end{proof}

We will now approximate the subsolution $\vphi$. Let us denote
\[
\vphi_\vepsilon := \max\{\vphi-\vepsilon, A \rho/\vepsilon\},
\]
where $A:= 1+ \|\vphi\|_\infty$. 

\begin{lem} \label{lem:sub-sol-mass-est}
We have \[
	\int_\Omega (dd^c\vphi_\vepsilon)^n \leq \frac{C A^n}{\vepsilon^{n}}. 
\]
Moreover, 
\[ {\bf 1}_{D_\vepsilon} \cdot \mu \leq (dd^c \vphi_\vepsilon)^n\] as two measures, where $D_\vepsilon$ is defined in \eqref{eq:d-domain}.
\end{lem}

\begin{proof}  To estimate the Monge-Amp\`ere mass of $\vphi_\vepsilon$ we use a result of Bedford and Taylor \cite[Corollary~4.3]{BT82} which is a consequence of the comparison principle.
Since $\frac{A \rho}{\vepsilon} \leq \vphi_\vepsilon \leq 0$ and the functions have the zero values on the boundary, 
\[
	\int_{\Omega} (dd^c\vphi_\vepsilon)^n 
	\leq  \frac{A^n}{\vepsilon^n}\int_{\Omega} (dd^c \rho)^n.
\]
The last integral is finite as  $\rho$ is $C^2$ on a neighborhood of the closure of $\Omega$. Furthermore,  since $\vphi_\vepsilon(z) = \vphi(z)-\vepsilon$ on $D_\vepsilon = \{\rho<-\vepsilon\}$ as $\vepsilon>0$ small, it is clear that $${\bf 1}_{D_\vepsilon} \cdot \mu \leq (dd^c \vphi_\vepsilon)^n.$$
This completes the proof of the lemma.
\end{proof}

\begin{remark} \label{rmk:mass-control}
Using the same argument we also get that for an integer $1\leq k \leq n$
\[
	\int_\Omega (dd^c\vphi_\vepsilon)^k \wed \beta^{n-k} \leq \frac{CA^k}{\vepsilon^k}.
\]
\end{remark}

We obtain now the volume-capacity inequality for the approximation sequence.
  
\begin{cor} \label{cor:vol-cap}
There exists  uniform constants $\alpha_1>0$ and $C>0$ which are independent of $\vepsilon$ such that for every compact set $K\subset \Omega$,
\[
	\int_K (dd^c\vphi_\vepsilon)^n \leq \frac{C}{\vepsilon^{n}}  \cdot cap(K) \cdot \exp\left(\frac{-\alpha_1}{[cap(K)]^\frac{1}{n}}\right).
\]
In particular, for a fixed $\tau>0$, there is a constant $C(\tau)>0$ such that
for every compact set $K\subset \Omega$,
\[
	\int_K (dd^c\vphi_\vepsilon)^n \leq \frac{C(\tau)}{\vepsilon^{n}} \left[cap(K)\right]^{1+\tau}.
\]
\end{cor}

\begin{proof} This is the analogue of Proposition~\ref{prop:vol-cap} with $\vphi$ is replaced by $\vphi_\vepsilon.$ The proof is a repetition of the one of this proposition. Here we need to take into account three facts: 
\[
	\|\vphi_\vepsilon\|_{\infty} \leq \frac{C}{\vepsilon} \quad\mbox{and}\quad \|\vphi_\vepsilon\|_{C^{0,\alpha}(\bar\Omega)} \leq \frac{C}{\vepsilon}, 
\]
and for an integer $1\leq k \leq n$ (Remark~\ref{rmk:mass-control}),
\[ 
	\int_\Omega (dd^c\vphi_\vepsilon)^k \wed \beta^{n-k} \leq \frac{C}{\vepsilon^{k}}. 
\]
This explains why we need an extra factor $C/\vepsilon^{n}$ on the right hand side of the inequality.
\end{proof}

Next, we have the following stability estimate for the Monge-Amp\`ere equation similar to \cite[Theorem~1.1]{GKZ08}. However, it also takes into account the possibility of infinite total mass of the measure on the right hand side.

\begin{prop} \label{prop:stability} Let $u$ be the solution of the equation \eqref{eq:sol-B} and $\Omega_\vepsilon$ is defined by \eqref{eq:omega-epsilon}. 
Let $v\in PSH\cap L^\infty(\Omega)$ be such that  $v= u$ on $\Omega\setminus \Omega_\vepsilon$. Then, there is $0<\alpha_2\leq 1$ such that
\[
	\sup_{\Omega} (v - u) \leq \frac{C}{\vepsilon^{n}} \left(\int_{\Omega} \max\{v -u, 0\} d\mu\right)^{\alpha_2}.
\]
\end{prop}

\begin{proof} Without loss of generality we may assume that $ \sup_{\Omega} (v-u)>0$. Set \[s_0 := \inf_{\Omega} (u-v).\] We know that
for $0<s < |s_0|$,
\[
	U(s):=\{u< v + s_0 +s\} \subset \subset \Omega_\vepsilon.
\]
\begin{lem}\label{lem:capacity-growth}
Fix $\tau>0$. For every $0< s, t < \frac{|s_0|}{2}$. Then,
\[
	t^n cap(U(s)) \leq \frac{C(\tau)}{\vepsilon^{n}} \left[ cap(U(s+t)) \right]^{1+\tau}.
\]
\end{lem}
\begin{proof}[Proof of Lemma~\ref{lem:capacity-growth}]
Let $0 \leq w \leq 1$ be a plurisubharmonic function in $\Omega$. We have the following chain of inequalites.
\[\begin{aligned}
	t^n \int_{U(s)} (dd^cw)^n 
&= 		\int_{\{u< v+s_0+s\}} [dd^c (tw)]^n \\
&\leq		\int_{\{u<v+s_0 +s+tw\}} 	[dd^c (v + tw)]^n \\
&\leq		\int_{\{u<v+s_0 +s+tw\}} 	(dd^c u)^n,  \\
\end{aligned}\]
where we used the comparison principle \cite[Theorem~4.1]{BT82} for the last inequality.  Since $\{u<v+s_0 +s+tw\} \subset U(s+t)$ and $w$ is arbitrary, we get that 
\[
	t^n cap(U(s)) \leq \int_{U(s+t)} d\mu.
\]
If we denote  $\vepsilon' := c_0\vepsilon$, where $c_0$ is the constant in \eqref{eq:dist-bound}, then
$${\bf 1}_{D_{\vepsilon'}} \cdot d\mu \leq (dd^c\vphi_{\vepsilon'})^n$$ 
as two measures. 
Since $U(s+t) \subset \Omega_\vepsilon \subset D_{\vepsilon'}$, it follows that
\[\begin{aligned}
	\int_{U(s+t)} d\mu 
&\leq		 \int_{U(s+t)} (dd^c\vphi_{\vepsilon'})^n \\
&\leq		\frac{C(\tau)}{(c_0\vepsilon)^{n}} \left[ cap(U(s+t)) \right]^{1+\tau},
\end{aligned}\]
where the last inequality followed from Corollary~\ref{cor:vol-cap}. The proof of the lemma is complete.
\end{proof}
Now together with Lemma~\ref{lem:capacity-growth}, the rest of the proof of the proposition is the same as in \cite[Theorem~1.1]{GKZ08} (see also \cite[Theorem~3.11]{KN3}). 
\end{proof}

The following result is a variation of Lemma~2.7 in \cite{Ng17a} where we considered the H\"older continuity of a measure $\nu$ on $\cE_0'$. Though, the situation now is differrent as $\nu(\Omega)$ is no longer finite.

\begin{thm} \label{thm:l1-l1}
Let $u$ be the solution of the equation \eqref{eq:sol-B} and $\Omega_\vepsilon$ is defined by \eqref{eq:omega-epsilon}.
Let $v \in PSH \cap L^\infty(\Omega)$ be such that  that $v=u$ on $\Omega\setminus\Omega_\vepsilon$. Then, there exists $0<\alpha_3\leq 1$ such that
\[
	\int_\Omega |v-u| d\mu \leq \frac{C}{\vepsilon^{n+1}} \left(\int_\Omega |v-u| dV_{2n} \right)^{\alpha_3}.
\]
\end{thm}

\begin{proof} This is a variation of the inequality \eqref{eq:holder-cegrell-class} with 
\[
	S_{k,\vepsilon} := (dd^c\vphi_\vepsilon)^k \wed \beta^{n-k},
\]
where $\vphi_\vepsilon= \max\{\vphi - \vepsilon, A\rho/\vepsilon\}$ and  $0\leq k \leq n$ is integer. Since $\mu \leq S_{n,\vepsilon}$ on $\Omega_\vepsilon$, it is enough to show that there is $0<\tau \leq 1$ satisfying
\[\label{eq:thm-l1-l1-2}
	\int_\Omega (v-u) S_{n,\vepsilon}  \leq \frac{C}{\vepsilon^{n+1}} \|v-u\|_1^{\tau}.
\]
for $v\geq u$ on $\Omega$. (In the general case we use the  identity
\[\notag
	|v-u| = (\max\{v,u\} -u) + (\max\{v,u\} -v)
\]
and apply twice the inequality \eqref{eq:thm-l1-l1-2} to get the theorem.)

Now we can repeat the inductive arguments of the proof of \eqref{eq:holder-cegrell-class} with $v, u$ and $\vphi_\vepsilon$ in the places of $v_s, v$ and $\vphi$, respectively. However, there are differences as follows. First, $v, u$ are no longer in $\cE_0'$. Second, let us extend $\vphi$ as in proof of Lemma~\ref{lem:exp-decay}, then $\vphi_\vepsilon= \max\{\vphi - \vepsilon, A\rho/\vepsilon\}$ is also defined on the neighborhood $U$ of $\bar\Omega$, and
$$	
\|\vphi_\vepsilon\|_{C^{0,\alpha}(U)} \leq \frac{C}{\vepsilon}.
$$
Taking into account above differences, to pass from the  $k$-th step to the step number $(k+1)$ we need the following inequality,
corresponding to   \eqref{eq:i1+i2},  (with notation $S_\vepsilon:= (dd^c \vphi_\vepsilon)^{k} \wed \beta^{n-k-1}$)
\[\label{eq:ind-est}\begin{aligned}
	\int_\Omega (v-u) dd^c \vphi_\vepsilon \wed S_\vepsilon 
&\leq 	\left| \int_\Omega (v-u) dd^c \vphi_\vepsilon * \chi_t \wed S_\vepsilon \right| \\
&\quad	+\left| \int_\Omega (v-u) dd^c (\vphi_\vepsilon * \chi_t -\vphi_\vepsilon) \wed S_\vepsilon\right| \\
&=: I_{1,\vepsilon}+ I_{2,\vepsilon}.
\end{aligned}\]
Since \[\label{eq:observe-holder-b}
	\left|\frac{\d^2 \vphi_\vepsilon * \chi_t}{\d z_j \d \bar z_k} (z)\right| \leq \frac{C \|\vphi\|_\infty}{\vepsilon t^2},
\]
and the induction hypothesis at the step $k$-th, there exists $0<\tau_k \leq 1$ such that
$$
\int_\Omega (v-u) S_\vepsilon \wed \beta \leq \frac{C}{\vepsilon^{k+1}} \|v-u\|_1^{\tau_k},
$$
we conclude that
\[\label{eq:ind-est-a}
\begin{aligned} I_{1,\vepsilon}
&\leq 	\frac{C\|\vphi\|_\infty}{\vepsilon\, t^2} \int_\Omega (v-u) S_\vepsilon \wed \beta \\
&\leq 	\frac{C\|\vphi\|_\infty}{\vepsilon^{k+2}\, t^2} \|v-u\|_1^{\tau_k}.
\end{aligned}\]
Similar to \eqref{eq:i2-int}, by integration by parts, $u=v$ on $\Omega\setminus \Omega_\vepsilon$, and 
\[\notag \label{eq:observe-holder-a}
\begin{aligned}
	\left|\vphi_\vepsilon * \chi_t (z) - \vphi_\vepsilon(z) \right|
&	\leq \frac{Ct^{\alpha}}{\vepsilon},
\end{aligned}\]
it follows that 
\[\label{eq:i2e-a}\begin{aligned}
	I_{2,\vepsilon} 
&\leq		\frac{C t^\alpha}{\vepsilon} \int_{\Omega_\vepsilon} (dd^c v+ dd^cu) \wed S_\vepsilon. \\
\end{aligned}\]
At this point as $u, v$ do not belong to $\cE_0'$ we need to use a different argument to bound $I_{2,\vepsilon}$.  Namely,  similarly to Lemma~\ref{lem:mass-est}, we have
\[\label{eq:i2e-b}
\int_{\Omega_\vepsilon} (dd^c u + dd^c v) \wed S_\vepsilon  \leq  \frac{C\|u+v\|_\infty (1+\|\vphi\|_\infty)^k}{\vepsilon^{k+1}} .
\]
Indeed, we first have 
\[\notag\begin{aligned}
&\int_{\Omega_\vepsilon} dd^c (u+v) \wed (dd^c \vphi_\vepsilon)^k\wed \beta^{n-k-1}  \\
&\leq		\frac{2}{\vepsilon'}\int_{\Omega} \left(\max\{\rho, - \vepsilon'/2\} - \rho\right) \wed dd^c (u+v) \wed (dd^c \vphi_\vepsilon)^k\wed \beta^{n-k-1} \\
&\leq \frac{C}{\vepsilon} \|u+v\|_\infty  \int_\Omega (dd^c\rho)\wed (dd^c \vphi_\vepsilon)^k\wed \beta^{n-k-1},
\end{aligned}\]
where $\vepsilon' = c_0\vepsilon$ with $c_0$ defined by \eqref{eq:dist-bound}. 
The desired inequality \eqref{eq:i2e-b} follows from Remark~\ref{rmk:mass-control}. 
Now, combining \eqref{eq:i2e-a} and \eqref{eq:i2e-b}  we get that
\[\label{eq:ind-est-b}
I_{2,\vepsilon}  \leq		\frac{C t^\alpha}{\vepsilon^{k+2}}.
\]
Next, it is easy to see (from Lemma~\ref{lem:sub-sol-mass-est}) that 
\[\notag
	\int_{\Omega} (v-u) S_n \leq \frac{C \|u\|_\infty (1+ \|\vphi\|_\infty)^n}{\vepsilon^{n}}. 
\]
Therefore, we can assume that $0< \|v-u\|_1 < 0.01$.
Thanks to \eqref{eq:ind-est-a} and \eqref{eq:ind-est-b} we have
\[\notag
	\int_\Omega (v-u) dd^c \vphi_\vepsilon\wed S_\vepsilon \leq \frac{C}{\vepsilon^{k+2} \, t^2} \|v-u\|_1^{\tau_k} + \frac{C t^\alpha}{\vepsilon^{k+2}}.
\]
If we choose 
$
	t = \|v-u\|_1^\frac{\tau_k}{3}, 
	\quad \tau_{k+1} = \frac{\alpha\tau_k}{3},
$
then
\[\notag
	\int_\Omega (v-u) S_\vepsilon \wed dd^c \vphi_\vepsilon \leq \frac{C}{\vepsilon^{k+2}} \|v-u\|_1^{\tau_{k+1}}.
\]
Thus, the induction argument is completed, and the theorem follows. 
\end{proof}

The last ingredient to prove Theorem~B was proved first in  \cite{BKPZ16} (see also \cite[Lemma~2.12]{Ng17a}). Here, the estimate is sharper and the proof is simpler too.

\begin{lem}  \label{lem:lap-est}
For $\delta>0$ small we have
\[
	\int_{\Omega_\delta} |\hat u_\delta -u | dV_{2n} \leq C \delta.
\]
\end{lem}

\begin{proof} First, we know from the classical Jensen formula (see e.g. \cite[Lemma 4.3]{GKZ08}) that
\[
	\int_{\Omega_{2\delta}} |\hat u_\delta -u| \leq C \delta^2 \int_{\Omega_\delta} \Delta u(z).
\]
Again, it follows from Lemma~\ref{lem:mass-est} applied for $k =1$ and $\delta = \vepsilon$, that
\[
	\int_{\Omega_\delta} \Delta u(z) \leq \frac{C}{\delta}.
\]
Therefore,
\[
	\int_{\Omega_\delta} |\hat u_\delta -u| dV_{2n} \leq \int_{\Omega_{2\delta}} |\hat u_\delta -u| dV_{2n} + \|u\|_\infty \int_{\Omega_\delta\setminus \Omega_{2\delta}} dV_{2n} \leq C\delta.
\]
This is the required inequality.
\end{proof}

We are ready to prove the H\"older continuity of the solution.

\begin{proof}[End of Proof of Theorem~B] Let us fix $\delta$ such that $0< \delta < \delta_0$ small and let $\vepsilon$ be such that $ \delta \leq \vepsilon < \delta_{0}$ which is to be determined later. 
Thanks to  Lemma~\ref{lem:boundary-holder} and $\hat u_\delta \leq u_\delta$ we have $\hat u_\delta - C\vepsilon^\alpha \leq u$ on $\d\Omega_\vepsilon$. Therefore, the function
\[ \label{eq:extend-solution}
\tilde u := 
\begin{cases} 
	\max\{\hat u_\delta - C \vepsilon^\alpha, u\} \quad 
	&\mbox{ on } \Omega_{\vepsilon},\\
	u  \quad &\mbox{ on } \Omega\setminus \Omega_{\vepsilon},
\end{cases}
\]
belongs to $PSH(\Omega)\cap C^0(\bar\Omega)$. Notice that $\tilde u \geq u$ in $\Omega$, and
\[
	 \tilde u = u \quad\mbox{ on } \Omega\setminus \Omega_\vepsilon.
\]
Again, by the second part of Lemma~\ref{lem:boundary-holder}  we have that
\[\label{eq:holder-eq1}
\begin{aligned}
	\sup_{\Omega_\delta}(\hat u_\delta -u) 
&\leq 	\sup_{\Omega_\vepsilon} (\hat u_\delta -u) + C \vepsilon^\alpha \\
&\leq		\sup_{\Omega} (\tilde u - u) + C \vepsilon^\alpha + C\vepsilon^\alpha.
\end{aligned}\]
By the stability estimate (Proposition~\ref{prop:stability}) there exists $0<\alpha_2 \leq 1$ such that
\[\label{eq:holder-eq2}\begin{aligned}
	\sup_{\Omega} (\tilde u - u) 
&\leq 	\frac{C}{\vepsilon^{n}} \left(\int_{\Omega} \max\{\tilde u - u,0\} d\mu\right)^{\alpha_2} \\
&\leq		\frac{C}{\vepsilon^{n}} \left(\int_{\Omega} |\tilde u - u|  d\mu\right)^{\alpha_2},	
\end{aligned}\]
where we used the fact that $\tilde u = u$ outside $\Omega_\vepsilon$.
Using Theorem~\ref{thm:l1-l1}, there is $0<\alpha_3\leq 1$ such that
\[\label{eq:holder-eq3}
\begin{aligned}
	\sup_{\Omega} (\tilde u -u) 
&\leq 	\frac{C}{\vepsilon^{n+(n+1)\alpha_2}}  \left(\int_{\Omega} |\tilde u - u|  dV_{2n}\right)^{\alpha_2\alpha_3} \\
&\leq		\frac{C}{\vepsilon^{2n+1}}  \left(\int_{\Omega_\delta} |\hat u_\delta - u|  dV_{2n}\right)^{\alpha_2\alpha_3}, 
\end{aligned}\]
where we used $0\leq \tilde u - u \leq {\bf 1}_{\Omega_\vepsilon} \cdot (\hat u_\delta -u)$ and $\Omega_\vepsilon \subset \Omega_{\delta}$ for the second inequality. It follows from \eqref{eq:holder-eq1}, \eqref{eq:holder-eq3} and Lemma~\ref{lem:lap-est} that
\[
	\sup_{\Omega_\delta} (\hat u_\delta - u) \leq C \vepsilon^\alpha + \frac{C \delta^{\alpha_2\alpha_3}}{\vepsilon^{2n+1}}.
\]
Now, we choose $\alpha_4 = \alpha\alpha_2\alpha_3/(2n+1+\alpha)$ and 
$$ \vepsilon = \delta^{\frac{\alpha_2\alpha_3 }{2n+1 + \alpha}}.
$$
Then, 
$
	\sup_{\Omega_\delta} (\hat u_\delta - u) \leq  C \delta^{\alpha_4}.
$
Finally, thanks to \cite[Lemma~4.2]{GKZ08} we infer that 
$
\sup_{\Omega_\delta} (u_\delta - u) \leq C \delta^{\alpha_4}.
$
The proof of the theorem is finished.
\end{proof}

\section{proof of Corollary~C}

Let $\mu\in \cM(\vphi,\Omega)$ and $0\leq f\in L^p(\Omega, d\mu)$ with $p>1$. We wish to show that there exists $\tilde\vphi \in PSH(\Omega) \cap C^{0,\tilde\alpha}(\bar\Omega)$ with $0<\tilde\alpha \leq 1$ such that
\[
	f d\mu \in \cM(\tilde\vphi,\Omega).
\]
The proof of the corollary is similar to the one of Theorem~B with the aid of following two lemmas. 

\begin{lem}  Fix a constant $\tau>0$. Then, there exists a uniform constant $C(\tau)$ such that for every compact set $K\subset \Omega$,
\[
	\int_K fd\mu \leq C(\tau) \left[cap(K)\right]^{1+\tau}.
\]
\end{lem}

\begin{proof}  H\"older's inequality and Proposition~\ref{prop:vol-cap} give us 
\[\begin{aligned}
\int_K fd\mu 
&\leq 	\|f\|_{L^p(\Omega, d\mu)} \left[\mu(K)\right]^\frac{p-1}{p} \\
&\leq 	C \left[cap(K) \cdot \exp\left(\frac{-\alpha_0}{[cap(K)]^\frac{1}{n}}\right)\right]^\frac{p-1}{p}.
\end{aligned}
\]
Let $0< a,b,c <1$ be fixed. The following elementary inequality holds for $x>0$,
$$x^a \exp\left(\frac{-c}{x^b}\right) \leq C(\tau) x^{1+\tau},$$
where $C(\tau) = C(\tau,a,b,c)$ depends only on $\tau, a, b, c$. Thus, the desired inequality follows.
\end{proof}
Thanks to the lemma and \cite[Theorem~5.9]{ko05} we can solve the Monge-Amp\`ere equation 
\[
	u\in PSH(\Omega) \cap C^0(\bar\Omega), \quad
	(dd^cu)^n = fd\mu, \quad u_{|_{\d\Omega}} =0.
\]
Moreover, the above lemma will enable us to have the stability estimate (Proposition~\ref{prop:stability}). 
The next lemma is also a consequence of the generalized H\"older inequality which was proved in \cite[Corollary~2.14]{Ng17a}. 

\begin{lem} \label{lem:lp-property}
Let $v \in PSH(\Omega) \cap C^0(\bar\Omega)$ be such that $v\geq u$ in $\Omega$ and $v= u$ near $\d\Omega$. Then, there exist uniform constants $C>0$ and $0< \tilde\alpha_3 <1$ such that 
\[
	\int_\Omega (v-u) f d\mu \leq C \|v-u\|_{L^1(d\mu)}^{\tilde\alpha_3}.
\]
\end{lem}
Next, we use the extendability assumption of $\vphi$ to get the one similar to Lemma~\ref{lem:boundary-holder} in the current setting. Namely, let $\tilde\Omega$ be a striclty pseudoconvex neighborhood of $\bar\Omega$ such that $\vphi \in PSH(\tilde\Omega)$ and H\"older continuous on the closure of $\tilde\Omega$. Thanks to the results in \cite{Ng17a} there exists  $v \in PSH(\tilde\Omega)$ and H\"older continuous in $\tilde\Omega$ satisfying
\[\notag
	(dd^c v)^n = {\bf 1}_{\Omega} f d \mu \quad \mbox{in } \tilde\Omega, \quad
	v = 0 \quad \mbox{on } \d\tilde\Omega.
\]
Consider $h$ to be the maximal pluriharmonic extension into $\Omega$ of $(- v)_{|_{\d\Omega}}$  which is H\"older continuous on $\d\Omega$. So is  $h$ on $\bar\Omega$. Then, by the comparison principle,
\[\notag
	v + h \leq u \leq 0 \quad \mbox{ on } \bar\Omega.
\]
From this we easily deduce the desired estimate near boundary for $u$.

Now the rest of the proof goes exactly as in the proof of Theorem~B. Namely, the inequality \eqref{eq:holder-eq2} holds for the measure $fd\mu$, next use Lemma~\ref{lem:lp-property} and Theorem~\ref{thm:l1-l1} to get the inequality \eqref{eq:holder-eq3}. Then we get the H\"older continuity of $u$. Notice that the H\"older exponent is worse by a factor $\tilde\alpha_3$. Thus, $fd\mu\in \cM(u,\Omega)$.

\end{document}